\documentclass[12pt,twoside]{amsart}
\usepackage{amsmath}
\usepackage{amsthm}
\usepackage{amsfonts}
\usepackage{amssymb}
\usepackage{latexsym}
\usepackage{mathrsfs}
\usepackage{amsmath}
\usepackage{amsthm}
\usepackage{amsfonts}
\usepackage{amssymb}
\usepackage{latexsym}
\usepackage{geometry}
\usepackage{dsfont}
\usepackage[dvips]{graphicx}
\usepackage{color}
\usepackage[all]{xy}

\date{}
\allowdisplaybreaks[4] \footskip=15pt
\renewcommand{\uppercasenonmath}[1]{}

\usepackage{graphicx,amssymb}
\usepackage[all]{xy}
\usepackage{amsmath}

\allowdisplaybreaks
\usepackage{amsthm}
\usepackage{color}

\theoremstyle{plain}
\newtheorem{theorem}{Theorem}[section]
\newtheorem{proposition}[theorem]{Proposition}
\newtheorem{lemma}[theorem]{Lemma}
\newtheorem{corollary}[theorem]{Corollary}
\newtheorem{example}[theorem]{Example}
\newtheorem*{open question}{Open Question}
\newtheorem{definition}[theorem]{Definition}

\theoremstyle{definition}

\theoremstyle{remark}
\newtheorem{remark}[theorem]{Remark}

\newcommand{\Tor}{\mbox{\rm Tor}}

\newcommand{\C}{\mathcal{C}}

\def\p{\frak p}
\def\m{\frak m}

\def\Hom{{\rm Hom}}

\def\Tor{{\rm Tor}}

\def\Ker{{\rm Ker}}

\def\Im{{\rm Im}}
\def\Coker{{\rm Coker}}

\def\Ann{{\rm Ann}}

\def\Max{{\rm Max}}

\def\Spec{{\rm Spec}}
\def\Max{{\rm Max}}

\begin{document}
\begin{center}
{\large  \bf The $S$-weak global dimension of commutative rings}

\vspace{0.5cm}   Xiaolei Zhang

{\footnotesize  Department of Basic Courses, Chengdu Aeronautic Polytechnic, Chengdu 610100, China\\

E-mail: zxlrghj@163.com\\}
\end{center}

\bigskip
\centerline { \bf  Abstract}
\bigskip
\leftskip10truemm \rightskip10truemm \noindent

In this paper, we introduce and study the $S$-weak global dimension $S$-w.gl.dim$(R)$ of a commutative ring $R$ for some multiplicative subset $S$ of $R$. Moreover, commutative rings with  $S$-weak global dimension at most $1$ are studied. Finally, we investigated the $S$-weak global dimension of factor rings and  polynomial rings.
\vbox to 0.3cm{}

\leftskip0truemm \rightskip0truemm
\bigskip

Throughout this article, $R$ always is a commutative ring with unity $1$ and $S$ always is a multiplicative subset of $R$, that is, $1\in S$ and $s_1s_2\in S$ for any $s_1\in S, s_2\in S$. In 2002,  Anderson and Dumitrescu \cite{ad02} defined $S$-Noetherian rings $R$ for which any ideal of $R$ is $S$-finite. Recall from \cite{ad02} that an $R$-module $M$ is called $S$-finite provided that $sM\subseteq F$ for some $s\in S$ and some finitely generated submodule $F$ of $M$. An $R$-module $T$ is called uniformly $S$-torsion if $sT=0$ for some $s\in S$ in \cite{zwz21}. So  an $R$-module $M$ is  $S$-finite if and only if $M/F$ is uniformly $S$-torsion for some finitely generated submodule $F$ of $M$. The idea derived from  uniformly $S$-torsion modules  is deserved to be further investigated. In \cite{zwz21}, the author of this paper introduced the class of $S$-flat modules $F$ for which the functor $F\otimes_R-$ preserves $S$-exact sequences. The class of  $S$-flat modules can be seen as a ``uniform'' generalization of that of flat modules, since an $R$-module $F$ is $S$-flat if and only if  $\Tor^R_1(F,M)$ is  uniformly $S$-torsion for any  $R$-module $M$ (see \cite[Theorem 3.2]{zwz21}). The class of  $S$-flat modules owns the following $S$-hereditary property: let $0\rightarrow A\xrightarrow{f} B\xrightarrow{g} C\rightarrow 0$ be an  $S$-exact sequence, if $B$ and $C$ are  $S$-flat so is $A$ (see \cite[Proposition 3.4]{zwz21}). So it is worth to study the  $S$-analogue of flat dimensions of $R$-modules and $S$-analogue of weak global dimension of commutative rings.

In this article, we define the $S$-flat dimension $S$-$fd_R(M)$ of an  $R$-module $M$ to be the length of the shortest $S$-flat $S$-resolution of $M$.  We characterize  $S$-flat dimensions of $R$-modules using the  uniform torsion property of  the ``Tor''  functors  in Proposition  \ref{w-g-flat}. Besides,  we obtain a new local characterization of flat dimensions of $R$-modules (see Corollary \ref{wgld-swgld}).  The $S$-weak global dimension $S$-w.gl.dim$(R)$ of a commutative ring $R$  is defined to be  the supremum of $S$-flat dimensions of all $R$-modules. A characterization of $S$-weak global dimensions is given in Proposition \ref{w-g-flat}. Examples of rings $R$ for which $S$-w.gl.dim$(R)\not=$w.gl.dim$(R_S)$ can be found in Example \ref{s-wgld-1-not-wgld-1}. $S$-von Neumann regular rings are firstly introduced in \cite{zwz21}  for which  there exists  $s\in S$ such that  for any $a\in R$ there exists  $r\in R$ such that $sa=ra^2$. By \cite[Theorem 3.11]{zwz21}, a ring $R$ is $S$-von Neumann regular if and only if all $R$-modules are $S$-flat. So $S$-von Neumann regular rings are exactly commutative rings with $S$-weak global dimension equal to $0$ (see Corollary \ref{s-vn-ext-char}). We also study commutative rings $R$ with  $S$-w.gl.dim$(R)$ at most $1$.  The nontrivial example of commutative rings with $S$-w.gl.dim$(R)\leq 1$ but infinite weak global dimension is given in Example \ref{s-wgld-1-not-wgld-1}. In the final section, we investigate the $S$-weak global dimensions of factor rings and  polynomial rings and show that $S$-w.gl.dim$(R[x])=S$-w.gl.dim$(R)+1$ (see Theorem \ref{s-wgd-poly}).

\section{Preliminaries}

Recall from \cite{zwz21}, an $R$-module $T$ is called a uniformly $S$-torsion module  provided that there exists an element $s\in S$ such that $sT=0$.
An $R$-sequence  $M\xrightarrow{f} N\xrightarrow{g} L$ is called  \emph{$S$-exact} (at $N$) provided that there is an element $s\in S$ such that $s\Ker(g)\subseteq \Im(f)$ and $s\Im(f)\subseteq \Ker(g)$. We say a long $R$-sequence $...\rightarrow A_{n-1}\xrightarrow{f_n} A_{n}\xrightarrow{f_{n+1}} A_{n+1}\rightarrow...$ is $S$-exact, if for any $n$ there is an element $s\in S$ such that $s\Ker(f_{n+1})\subseteq \Im(f_n)$ and $s\Im(f_n)\subseteq \Ker(f_{n+1})$. An $S$-exact sequence $0\rightarrow A\rightarrow B\rightarrow C\rightarrow 0$ is called a short $S$-exact sequence. An $R$-homomorphism $f:M\rightarrow N$ is an \emph{$S$-monomorphism}  $($resp.,   \emph{$S$-epimorphism}, \emph{$S$-isomorphism}$)$ provided $0\rightarrow M\xrightarrow{f} N$   $($resp., $M\xrightarrow{f} N\rightarrow 0$, $0\rightarrow M\xrightarrow{f} N\rightarrow 0$ $)$ is   $S$-exact.
It is easy to verify an  $R$-homomorphism $f:M\rightarrow N$ is an $S$-monomorphism $($resp., $S$-epimorphism, $S$-isomorphism$)$ if and only if $\Ker(f)$ $($resp., $\Coker(f)$, both $\Ker(f)$ and $\Coker(f)$$)$ is a  uniformly $S$-torsion module.

\begin{proposition}\label{s-iso-inv}
Let $R$ be a ring and $S$ a multiplicative subset of $R$. Suppose there is an $S$-isomorphism $f:M\rightarrow N$ for  $R$-modules  $M$ and $N$. Then there is an $S$-isomorphism $g:N\rightarrow M$.
\end{proposition}
\begin{proof} Consider the following commutative diagram:
$$\xymatrix{
0\ar[r]^{}& \Ker(f)\ar[r]^{} &M\ar[rr]^{f}\ar@{->>}[rd] &&N \ar[r]^{} & \Coker(f)\ar[r]^{} &  0\\
  & & &\Im(f) \ar@{^{(}->}[ru] &&  &   \\}$$
with $s\Ker(f)=0$ and $sN\subseteq \Im(f)$ for some $s\in S$. Define $g_1:N\rightarrow \Im(f)$ where $g_1(n)=sn$ for any $n\in N$. Then $g_1$ is a well-defined $R$-homomorphism since $sn\in \Im(f)$. Define $g_2:\Im(f)\rightarrow M$ where $g_2(f(m))=sm$. Then  $g_2$ is  well-defined $R$-homomorphism. Indeed, if $f(m)=0$, then $m\in \Ker(f)$ and so $sm=0$. Set $g=g_2\circ g_1:N\rightarrow M$. We claim that $g$ is an  $S$-isomorphism. Indeed, let $n$ be an element in $\Ker(g)$. Then $sn=g_1(n)\in \Ker(g_2)$. Note that $s\Ker(g_2)=0$. Thus $s^2n=0$. So $s^2\Ker(g)=0$. On the other hand, let $m\in M$. Then $g(f(m))=g_2\circ g_1(f(m))=g_2(f(sm))=s^2m$. Thus $s^2m\in \Im(g)$. So $s^2M\subseteq \Im(g)$. It follows that $g$ is  an $S$-isomorphism.
\end{proof}

\begin{remark}
Let $R$ be a ring, $S$ a multiplicative subset of $R$ and $M$ and $N$ $R$-modules. Then the condition ``there is an $R$-homomorphism $f:M\rightarrow N$  such that $f_S:M_S\rightarrow N_S$ is an isomorphism''  does not mean ``there is an $R$-homomorphism $g:N\rightarrow M$  such that $g_S:N_S\rightarrow M_S$ is an isomorphism''.

Indeed, let $R=\mathbb{Z}$ be the ring of integers, $S=R-\{0\}$  and $\mathbb{Q}$ the quotient field of integers. Then the embedding map $f:\mathbb{Z}\hookrightarrow \mathbb{Q}$ satisfies $f_S:\mathbb{Q}\rightarrow \mathbb{Q}$ is an isomorphism. However, since $\Hom_\mathbb{Z}(\mathbb{Q},\mathbb{Z})=0$, there does not exist any $R$-homomorphism $g:\mathbb{Q}\rightarrow \mathbb{Z}$  such that $g_S:\mathbb{Q}\rightarrow \mathbb{Q}$ is an isomorphism.
\end{remark}

Let $R$ be a ring and  $S$ a multiplicative subset of $R$. Suppose $M$ and $N$ are $R$-modules. We say $M$ is $S$-isomorphic to $N$ if there exists an $S$-isomorphism $f:M\rightarrow N$. A family $\C$  of $R$-modules  is said to be closed under $S$-isomorphisms if $M$ is $S$-isomorphic to $N$ and $M$ is in $\C$, then $N$ is  also in  $\C$. It follows from Proposition \ref{s-iso-inv} that the existence of $S$-isomorphisms of two $R$-modules is an equivalence relation. Next, we give an $S$-analogue of Five Lemma.

\begin{theorem}{\bf ($S$-analogue of Five Lemma)}\label{s-5-lemma}
Let $R$ be a ring, $S$ a multiplicative subset of $R$. Consider the following diagram with $S$-exact rows:
$$\xymatrix@R=20pt@C=20pt{
A\ar[d]_{f_A} \ar[r]^{g_1} & B\ar[d]_{f_B}\ar[r]^{g_2} &C\ar[d]^{f_C}\ar[r]^{g_3} &D\ar[r]^{g_4}\ar[d]^{f_D}&E\ar[d]^{f_E} \\
A'\ar[r]^{h_1} &B'\ar[r]^{h_2}  & C'  \ar[r]^{h_3} & D' \ar[r]^{h_4} & E' .\\}$$
\begin{enumerate}
\item  If $f_B$ and $f_D$ are $S$-monomorphisms and $f_A$ is an $S$-epimorphism, then $f_C$ is an $S$-monomorphism.
\item If $f_B$ and $f_D$ are $S$-epimorphisms and $f_E$ is an $S$-monomorphism, then $f_C$ is an $S$-epimorphism.
\item If $f_A$ is an $S$-epimorphism, $f_E$ is an $S$-monomorphism, and $f_B$ and $f_D$ are $S$-isomorphisms, then $f_C$ is an $S$-isomorphism.
\item  If $f_A$, $f_B$, $f_D$ and $f_E$ are all $S$-isomorphisms, then $f_C$ is an $S$-isomorphism.
\end{enumerate}
\end{theorem}
\begin{proof} (1) Let $x\in \Ker(f_C)$. Then $f_Dg_3(x)=h_3f_C(x)=0$. Since  $f_D$ is an $S$-monomorphism, $s_1\Ker(f_D)=0$ for some $s_1\in S$. So $s_1g_3(x)=g_3(s_1x)=0$. Since the top row is $S$-exact, there exists $s_2\in S$ such that $s_2\Ker(g_3)\subseteq \Im(g_2)$. Thus there exists $b\in B$ such that $g_2(b)=s_2s_1x$. Hence $h_2f_B(b)=f_Cg_2(b)=f_C(s_2s_1x)=0$. Thus there exists $s_3\in S$ such that $s_3\Ker(h_2)\subseteq \Im(h_1)$. So there exists
$a'\in A'$ such that  $h_1(a')=s_2f_B(b)$. Since $f_A$ is an $S$-epimorphism, there exists $s_4\in S$ such that $s_4A'\subseteq \Im(f_A)$. So there there exists $a\in A$ such that $s_4a'=f_A(a)$. Hence $s_4s_2f_B(b)=s_4h_1(a')=h_1(f_A(a))=f_B(g_1(a))$. So $s_4s_2b-g_1(a)\in \Ker(f_B)$.  Since $f_B$ is an $S$-monomorphism, there exists $s_5\in S$ such that $s_5\Ker(f_B)=0$.  Thus  $s_5(s_4s_2b-g_1(a))=0$. So  $s_5s_4s_2s_2s_1x=s_5(g_2(s_4s_2b))=s_5g_2(g_1(a))$. Since the top row is $S$-exact at $B$, then there exists $s_6\in S$ such that $s_6\Im(g_1)\subseteq \Ker(g_2)$. So $s_6s_5s_4s_2s_2s_1x=s_5g_2(s_6g_1(a))=0$. Consequently, if we set $s=s_6s_5s_4s_2s_2s_1$, then  $s\Ker(f_C)=0$. It follows that $f_C$ is an $S$-monomorphism.

 (2) Let $x\in C'$. Since $f_D$ is an $S$-epimorphism, then there exists $s_1\in S$ such that $s_1D'\subseteq \Im(f_D)$. Thus there exists $d\in D$ such that $f_D(d)=s_1h_3(x)$. By the commutativity of the right square, we have $f_Eg_4(d)=h_4f_D(d)=s_1h_4(h_3(x))$. Since the bottom row is $S$-exact at $D'$, there exists $s_2\in S$ such that $s_4\Im(h_3)\subseteq \Ker(h_4)$. So $s_4f_E(g_4(d))=s_1h_4(s_4h_3(x))=0$. Since $f_E$ is an $S$-monomorphism, there exists $s_3\in S$ such that $s_3\Ker(f_E)=0$. Thus $s_3s_4g_4(d)=0$. Since the top row is $S$-exact at $D$, there there exists $s_5\in S$ such that $s_5\Ker(g_4)\subseteq \Im(g_3)$. So there exists $c\in C$ such that $s_5s_3s_4d=g_3(c)$. Hence $s_5s_3s_4f_D(d)=f_D(g_3(c))=h_3(f_C(c))$. Since $s_5s_3s_4f_D(d)=h_3(s_1s_5s_3s_4x)$, we have $f_C(c)-s_1s_5s_3s_4x\in \Ker(h_3)$. Since the bottom row is $S$-exact at $C'$, there exists $s_6\in S$ such that $s_6\Ker(h_3)\subseteq \Im(h_2)$. Thus there exists $b'\in B'$ such that $s_6(f_C(c)-s_1s_5s_3s_4x)=h_2(b')$. Since $f_B$ is an $S$-epimorphism, then there exists $s_7\in S$ such that $s_7B'\subseteq \Im(f_B)$. So $s_7b'=f_B(b)$ for some $b\in B$. Thus $f_C(g_2(b))=h_2(f_B(b))=s_7h_2(b')=s_7(s_6(f_C(c)-s_1s_5s_3s_4x))$. So $s_7s_6s_1s_5s_3s_4x=s_7s_6f_C(c)-f_C(g_2(b))=f_C(s_7s_6c-g_2(b))\in\Im(f_C)$. Consequently, if we set $s=s_7s_6s_1s_5s_3s_4$, then $sC'\subseteq \Im(f_C)$.  It follows that $f_C$ is an $S$-epimorphism.

It is easy to see (3) follows from (1) and (2), while (4) follows from (3).
\end{proof}

Recall from \cite[Definition 3.1]{zwz21} that
an $R$-module $F$ is called  $S$-flat provided that for any  $S$-exact sequence $0\rightarrow A\rightarrow B\rightarrow C\rightarrow 0$, the induced sequence $0\rightarrow A\otimes_RF\rightarrow B\otimes_RF\rightarrow C\otimes_RF\rightarrow 0$ is  $S$-exact. It is easy to verify that the class of  $S$-flat modules is closed under $S$-isomorphisms by the following result.
\begin{lemma}\label{s-flat-ext} \cite[Theorem 3.2]{zwz21}
Let $R$ be a ring, $S$ a multiplicative subset of $R$ and $F$ an $R$-module. The following assertions are equivalent:
\begin{enumerate}
\item  $F$ is  $S$-flat;

\item for any short exact sequence $0\rightarrow A\xrightarrow{f} B\xrightarrow{g} C\rightarrow 0$, the induced sequence $0\rightarrow A\otimes_RF\xrightarrow{f\otimes_RF} B\otimes_RF\xrightarrow{g\otimes_RF} C\otimes_RF\rightarrow 0$ is  $S$-exact;

\item  $\Tor^R_1(M,F)$ is  uniformly $S$-torsion for any  $R$-module $M$;

\item  $\Tor^R_n(M,F)$ is  uniformly $S$-torsion for any  $R$-module $M$ and $n\geq 1$.

\end{enumerate}
\end{lemma}

The following result says that a short $S$-exact sequence induces a long $S$-exact sequence by the functor ``Tor'' as the classical case.

\begin{theorem}\label{s-iso-tor}
Let $R$ be a ring, $S$ a multiplicative subset of $R$ and $N$  an $R$-module. Suppose $0\rightarrow A\xrightarrow{f} B\xrightarrow{g} C\rightarrow 0$ is an $S$-exact sequence of $R$-modules. Then for any $n\geq 1$ there is an $R$-homomorphism $\delta_n:\Tor_{n}^R(C,N)\rightarrow \Tor_{n-1}^R(A,N)$ such that  the induced sequence
$$...\rightarrow \Tor_{n}^R(A,N)\rightarrow \Tor_{n}^R(B,N)\rightarrow \Tor_{n}^R(C,N)\xrightarrow{\delta_n} \Tor_{n-1}^R(A,N)\rightarrow $$
$$\Tor_{n-1}^R(B,N)\rightarrow ... \rightarrow\Tor_{1}^R(C,N)\xrightarrow{\delta_1} A\otimes_RN\rightarrow B\otimes_RN\rightarrow C\otimes_RN\rightarrow 0$$
 is $S$-exact.
\end{theorem}
\begin{proof}
Since the sequence $0\rightarrow A\xrightarrow{f} B\xrightarrow{g} C\rightarrow 0$ is $S$-exact at $B$. There are three exact sequences $0\rightarrow \Ker(f)\xrightarrow{i_{\Ker(f)}} A\xrightarrow{\pi_{\Im(f)}} \Im(f)\rightarrow 0$, $0\rightarrow \Ker(g)\xrightarrow{i_{\Ker(g)}} B\xrightarrow{\pi_{\Im(g)}} \Im(g)\rightarrow 0$ and $0\rightarrow \Im(g)\xrightarrow{i_{\Im(g)}} C\xrightarrow{\pi_{\Coker(g)}} \Coker(g)\rightarrow 0$ with $\Ker(f)$ and $\Coker(g)$ uniformly $S$-torsion. There also exists  $s\in S$ such that $s\Ker(g)\subseteq \Im(f)$ and  $s\Im(f)\subseteq \Ker(g)$. Denote $T=\Ker(f)$ and $T'=\Coker(g)$.

Firstly,  consider the exact sequence
$$\Tor_{n+1}^R(T',N)\rightarrow \Tor_{n}^R(\Im(g),N)\xrightarrow{\Tor_{n}^R(i_{\Im(g)},N)} \Tor_{n}^R(C,N)\rightarrow \Tor_{n}^R(T',N).$$ Since $T'$ is uniformly $S$-torsion, $\Tor_{n+1}^R(T',N)$ and $\Tor_{n}^R(T',N)$  is uniformly $S$-torsion. Thus $\Tor_{n}^R(i_{\Im(g)},N)$ is an $S$-isomorphism. So there is also an $S$-isomorphism $h^n_{\Im(g)}: \Tor_{n}^R(C,N)\rightarrow \Tor_{n}^R(\Im(g),N)$ by Proposition \ref{s-iso-inv}. Consider the exact sequence:
$$ \Tor_{n-1}^R(T,N)\rightarrow \Tor_{n-1}^R(A,N)\xrightarrow{\Tor_{n-1}^R(\pi_{\Im(f)},N)} \Tor_{n-1}^R(\Im(f),N)\rightarrow  \Tor_{n-2}^R(T,N).$$ Since $T$ is uniformly $S$-torsion, we have $\Tor_{n-1}^R(\pi_{\Im(f)},N)$ is an $S$-isomorphism. So there is also an $S$-isomorphism $h^{n-1}_{\Im(f)}: \Tor_{n-1}^R(\Im(f),N)\rightarrow \Tor_{n-1}^R(A,N)$ by Proposition \ref{s-iso-inv}.
We have two  exact sequences$$\Tor_{n+1}^R(T_1,N)\rightarrow \Tor_{n}^R(s\Ker(g),N)\xrightarrow{\Tor_{n}^R(i^1_{s\Ker(g)},N)}  \Tor_{n}^R(\Im(f),N)\rightarrow  \Tor_{n+1}^R(T_1,N)$$ and $$\Tor_{n+1}^R(T_2,N)\rightarrow \Tor_{n}^R(s\Ker(g),N)\xrightarrow{\Tor_{n}^R(i^2_{s\Ker(g)},N)}  \Tor_{n}^R(\Ker(g),N)\rightarrow  \Tor_{n+1}^R(T_2,N),$$  where $T_1=\Im(f)/s\Ker(g)$ and $T_2=\Im(f)/s\Im(f)$ is uniformly $S$-torsion. So $\Tor_{n}^R(i^1_{s\Ker(g)},N)$ and $\Tor_{n}^R(i^2_{s\Ker(g)},N)$ are $S$-isomorphisms. Thus there is an $S$-isomorphisms $h^n_{s\Ker(g)}: \Tor_{n}^R(\Ker(g),N)\rightarrow \Tor_{n}^R(s\Ker(g),N)$.
Note that there is an exact sequence $\Tor_{n}^R(B,N)\xrightarrow{\Tor_{n}^R(\pi_{\Im(g)},N)}\Tor_{n}^R(\Im(g),N)\xrightarrow{\delta^{n}_{\Im(g)}} \Tor_{n-1}^R(\Ker(g),N)$ $\xrightarrow{\Tor_{n-1}^R(i_{\Ker(g)},N)} \Tor_{n-1}^R(B,N).$
Set $\delta_n=h^n_{\Im(g)}\circ \delta^{n}_{\Im(g)}\circ h^n_{s\Ker(g)}\circ \Tor_{n}^R(i^1_{s\Ker(g)},N)\circ h^{n-1}_{\Im(f)} :\Tor_{n}^R(C,N)\rightarrow \Tor_{n-1}^R(A,N)$. Since $h^n_{\Im(g)}, \delta^{n}_{\Im(g)}, h^n_{s\Ker(g)}$ and $h^{n-1}_{\Im(f)}$ are $S$-isomorphisms, we have the sequence $\Tor_{n}^R(B,N)\rightarrow\Tor_{n}^R(C,N)\xrightarrow{\delta^{n}} \Tor_{n-1}^R(A,N)$ $\rightarrow \Tor_{n-1}^R(B,N)$ is $S$-exact.

Secondly, consider the exact sequence: $$\Tor_{n+1}^R(T,N)\rightarrow \Tor_{n}^R(A,N)\xrightarrow{\Tor_{n}^R(i_{\Im(f)},N)} \Tor_{n}^R(\Im(f),N)\rightarrow \Tor_{n}^R(T,N).$$
Since $T$ is uniformly $S$-torsion, $\Tor_{n}^R(i_{\Im(f)},N)$ is an $S$-isomorphism. Consider the exact sequences: $$\Tor_{n+1}^R(\Im(g),N)\rightarrow \Tor_{n}^R(\Ker(g),N)\xrightarrow{\Tor_{n}^R(i_{\Ker(g)},N)} \Tor_{n}^R(B,N)\rightarrow \Tor_{n}^R(\Im(g),N)$$
and $$\Tor_{n+1}^R(T',N)\rightarrow \Tor_{n}^R(\Im(g),N)\xrightarrow{\Tor_{n}^R(i_{\Im(g)},N)} \Tor_{n}^R(C,N)\rightarrow \Tor_{n}^R(T',N).$$
Since $T'$ is uniformly $S$-torsion, we have $\Tor_{n}^R(i_{\Im(g)}$ is an $S$-isomorphism.  Since $\Tor_{n}^R(i^1_{s\Ker(g)},N)$ and $\Tor_{n}^R(i^2_{s\Ker(g)},N)$ are $S$-isomorphisms as above, $\Tor_{n}^R(A,N)\rightarrow \Tor_{n}^R(B,N)\rightarrow \Tor_{n}^R(C,N)$ is $S$-exact at $\Tor_{n}^R(B,N)$.

Continue by the above method,  we have an  $S$-exact sequence:
$$...\rightarrow \Tor_{n}^R(A,N)\rightarrow \Tor_{n}^R(B,N)\rightarrow \Tor_{n}^R(C,N)\xrightarrow{\delta_n} \Tor_{n-1}^R(A,N)\rightarrow $$
$$\Tor_{n-1}^R(B,N)\rightarrow ... \rightarrow\Tor_{1}^R(C,N)\xrightarrow{\delta_1} A\otimes_RN\rightarrow B\otimes_RN\rightarrow C\otimes_RN\rightarrow 0.$$
\end{proof}

\begin{corollary}\label{big-Tor}
Let $R$ be a ring, $S$ a multiplicative subset of $R$ and $N$  an $R$-module. Suppose $0\rightarrow A\xrightarrow{f} B\xrightarrow{g} C\rightarrow 0$ is an $S$-exact sequence of $R$-modules where $B$ is $S$-flat. Then $\Tor^R_{n+1}(C,N)$ is $S$-isomorphic to  $\Tor_{n}^R(A,N)$ for any $n\geq 0$. Consequently,  $\Tor^R_{n+1}(C,N)$ is uniformly $S$-torsion if and only if $\Tor_{n}^R(A,N)$ is  uniformly $S$-torsion for any $n\geq 0$.
\end{corollary}
\begin{proof}
It follows from Lemma \ref{s-flat-ext} and Theorem \ref{s-iso-tor}.
\end{proof}

\section{On the $S$-flat dimensions of modules}

Let $R$ be a ring. The flat dimension of an $R$-module $M$ is defined as the shortest flat resolution of $M$. We now introduce the notion of $S$-flat dimension of an $R$-module as follows.

\begin{definition}\label{w-phi-flat }
Let $R$ be a ring, $S$ a multiplicative subset of $R$  and $M$ an $R$-module. We write $S$-$fd_R(M)\leq n$  $(S$-$fd$ abbreviates  \emph{$S$-flat dimension}$)$ if there exists an $S$-exact sequence of $R$-modules
$$ 0 \rightarrow F_n \rightarrow ...\rightarrow F_1\rightarrow F_0\rightarrow M\rightarrow 0   \ \ \ \ \ \ \ \ \ \ \ \ \ \ \ \ \ \ \ \ \ \ \ \ \ \ \ \ \ \ \ \ \ \ \ \ \ \ \ (\diamondsuit)$$
where each $F_i$ is  $S$-flat for $i=0,...,n$. The $S$-exact sequence $(\diamondsuit)$ is said to be an $S$-flat $S$-resolution of length $n$ of $M$. If such finite  $S$-flat  $S$-resolution does not exist, then we say $S$-$fd_R(M)=\infty$; otherwise,  define $S$-$fd_R(M)=n$ if $n$ is the length of the shortest $S$-flat $S$-resolution of $M$.
\end{definition}\label{def-wML}
Trivially, the $S$-flat dimension of an $R$-module $M$ cannot exceed its flat dimension for any multiplicative subset $S$ of $R$. And if $S$ is composed of units, then $S$-$fd_R(M)=fd_R(M)$.   It is also obvious that an $R$-module $M$ is $S$-flat if and only if $S$-$fd_R(M)= 0$.

\begin{lemma}\label{s-iso-fd}
Let $R$ be a ring, $S$ a multiplicative subset of $R$. If $A$ is $S$-isomorphic to $B$, then $S$-$fd_R(A) = S$-$fd_R(B)$.
\end{lemma}
\begin{proof}  Let $f: A\rightarrow B$ be an $S$-isomorphism. If $...\rightarrow F_n \rightarrow ...\rightarrow F_1\rightarrow F_0\xrightarrow{g} A\rightarrow 0$ is an $S$-resolution of $A$, then $...\rightarrow F_n \rightarrow ...\rightarrow F_1\rightarrow F_0\xrightarrow{f\circ g} B\rightarrow 0$ is an $S$-resolution of $B$. So $S$-$fd_R(A)\geq S$-$fd_R(B)$. Note that  there is an $S$-isomorphism  $g: B\rightarrow A$ by Proposition \ref{s-iso-inv}. Similarly we have $S$-$fd_R(B)\geq S$-$fd_R(A)$.
\end{proof}

\begin{proposition}\label{s-flat d}
Let $R$ be a ring and $S$ a multiplicative subset of $R$. The following statements are equivalent for an $R$-module $M$:
\begin{enumerate}
    \item $S$-$fd_R(M)\leq n$;
    \item $\Tor^R_{n+k}(M, N)$ is uniformly $S$-torsion for all  $R$-modules $N$ and all $k > 0$;
    \item $\Tor^R_{n+1}(M, N)$ is uniformly $S$-torsion for all $R$-modules $N$;
     \item there exists $s\in S$ such that $s\Tor^R_{n+1}(M, R/I)=0$  for all ideals $I$ of $R$;
      \item if $0 \rightarrow F_n \rightarrow ...\rightarrow F_1\rightarrow F_0\rightarrow M\rightarrow 0$ is an $S$-exact sequence, where $F_0, F_1, . . . , F_{n-1}$ are $S$-flat $R$-modules, then $F_n$ is $S$-flat;
      \item if $0 \rightarrow F_n \rightarrow ...\rightarrow F_1\rightarrow F_0\rightarrow M\rightarrow 0$ is an $S$-exact sequence, where $F_0, F_1, . . . , F_{n-1}$ are flat $R$-modules, then $F_n$ is $S$-flat;
        \item if $0 \rightarrow F_n \rightarrow ...\rightarrow F_1\rightarrow F_0\rightarrow M\rightarrow 0$ is an exact sequence, where $F_0, F_1, . . . , F_{n-1}$ are $S$-flat $R$-modules, then $F_n$ is $S$-flat;
    \item if $0 \rightarrow F_n \rightarrow ...\rightarrow F_1\rightarrow F_0\rightarrow M\rightarrow 0$ is an exact sequence, where $F_0, F_1, . . . , F_{n-1}$ are flat $R$-modules, then $F_n$ is $S$-flat;
          \item there exists an $S$-exact sequence $0 \rightarrow F_n \rightarrow ...\rightarrow F_1\rightarrow F_0\rightarrow M\rightarrow 0$, where $F_0, F_1, . . . , F_{n-1}$ are flat $R$-modules and $F_n$ is $S$-flat;
    \item there exists an exact sequence $0 \rightarrow F_n \rightarrow ...\rightarrow F_1\rightarrow F_0\rightarrow M\rightarrow 0$, where $F_0, F_1, . . . , F_{n-1}$ are flat $R$-modules and $F_n$ is $S$-flat;
         \item there exists an exact sequence $0 \rightarrow F_n \rightarrow ...\rightarrow F_1\rightarrow F_0\rightarrow M\rightarrow 0$, where $F_0, F_1, . . . , F_{n}$ are $S$-flat $R$-modules.
\end{enumerate}
\end{proposition}
\begin{proof}
$(1) \Rightarrow(2)$: We prove $(2)$ by induction on $n$. For the case $n = 0$, we have $M$ is $S$-flat, then (2) holds by \cite[Theorem 3.2]{zwz21}. If $n>0$, then
there is an $S$-exact sequence  $0 \rightarrow F_n \rightarrow ...\rightarrow F_1\rightarrow F_0\rightarrow M\rightarrow 0$,
where each $F_i$ is  $S$-flat for $i=0,...,n$. Set $K_0 = \Ker(F_0\rightarrow M)$ and $L_0=\Im(F_1\rightarrow F_0)$. Then both
$0 \rightarrow  K_0 \rightarrow  F_0 \rightarrow  M \rightarrow  0 $ and $0 \rightarrow  F_n \rightarrow  F_{n-1} \rightarrow...\rightarrow  F_1 \rightarrow  L_0 \rightarrow  0$ are $S$-exact. Since  $S$-$fd_R(L_0)\leq n-1$ and $L_0$ is $S$-isomorphic to $K_0$, $S$-$fd_R(K_0)\leq n-1$ by Lemma \ref{s-iso-fd}. By induction, $\Tor^R_{n-1+k}(K_0, N)$ is uniformly $S$-torsion for all $S$-torsion $R$-modules $N$ and all $k > 0$. It follows from Corollary \ref{big-Tor} that $\Tor^R_{n+k}(M, N)$ is uniformly $S$-torsion.

$(2) \Rightarrow(3)$, $(5)\Rightarrow(6)\Rightarrow(8)$ and $(5)\Rightarrow(7)\Rightarrow(8)$:  Trivial.

$(3) \Rightarrow(4)$: Let $N=\bigoplus\limits_{I\unlhd R}R/I$. Then there exists an element $s\in S$ such that $s\Tor^R_{n+1}(M,N)=0$. So $s\bigoplus\limits_{I\unlhd R}\Tor^R_{n+1}(M,R/I)=0$. It follows that $s\Tor^R_{n+1}(M, R/I)=0$  for all ideals $I$ of $R$.

$(4) \Rightarrow(3)$: Suppose $N$ is generated by $\{n_i|i\in \Gamma\}$.  Set $N_0=0$ and $N_\alpha=\langle n_i|i<\alpha\rangle$  for each $\alpha\leq \Gamma$.  Then $N$ have a continuous   filtration  $\{N_\alpha|\alpha\leq \Gamma \}$ with $N_{\alpha+1}/N_\alpha\cong R/I_{\alpha+1} $ and $I_{\alpha}=\Ann_R(n_{\alpha}+N_\alpha\cap Rn_{\alpha})$.  Since $s\Tor_1^R(M,R/I_{\alpha})=0$ for each $\alpha\leq \Gamma$, it is easy to verify $s\Tor_1^R(M,N_\alpha)=0$ by transfinite induction on $\alpha$. So $s\Tor_1^R(M,N)=0$.

$(3) \Rightarrow(5)$: Let $0 \rightarrow F_n \xrightarrow{d_n}F_{n-1} \xrightarrow{d_{n-1}}F_{n-2} ...\xrightarrow{d_2} F_1\xrightarrow{d_1} F_0\xrightarrow{d_0} M\rightarrow 0$ be an $S$-exact sequence,  where $F_0, F_1, . . . , F_{n-1}$ are $S$-flat.
Then $F_n$ is $S$-flat if and only if $\Tor_1^R(F_n,N)$ is  uniformly $S$-torsion for all  $R$-modules $N$, if and only if $\Tor_2^R(\Im(d_{n-1}),N)$ is  uniformly $S$-torsion for all  $R$-modules $N$. Following these steps, we can show $F_n$ is $S$-flat if and only if $\Tor^R_{n+1}(M, N)$ is uniformly $S$-torsion for all $R$-modules $N$.

$(10)\Rightarrow(11)\Rightarrow(1)$ and $(10)\Rightarrow(9)\Rightarrow(1)$:  Trivial.

$(8)\Rightarrow(10):$ Let $...\rightarrow P_n \rightarrow P_{n-1}\xrightarrow{d_{n-1}}  P_{n-2}... \rightarrow P_0\rightarrow M\rightarrow 0$ be a projective resolution of $M$. Set $F_n=\Ker(d_{n-1})$. Then we have an exact sequence  $0\rightarrow F_n\rightarrow P_{n-1}\xrightarrow{d_{n-1}}  P_{n-2}... \rightarrow P_0\rightarrow M\rightarrow 0$. By $(8)$, $F_n$ is $S$-flat. So $(10)$ holds.
\end{proof}

\begin{corollary}\label{sfd-sfd-1}
Let $R$ be a ring and $S'\subseteq S$  multiplicative subsets of $R$. Suppose $M$ is an $R$-module, then  $S$-$fd_R(M)\leq S'$-$fd_R(M)$.
\end{corollary}
\begin{proof} Suppose $S'\subseteq S$ are multiplicative subsets of $R$. Let  $M$ and $N$ be $R$-modules.  If  $\Tor^R_{n+1}(M, N)$ is uniformly $S'$-torsion, then  $\Tor^R_{n+1}(M, N)$ is uniformly $S$-torsion. The result follows by Proposition \ref{s-flat d}.
\end{proof}

Let $R$ be a ring, $S$ a multiplicative subset of $R$ and $M$ an $R$-module. For any $s\in S$, we denote by $R_s$ the localization of $R$ at $\{s^n|n\in\mathbb{N}\}$ and denote $M_s=M\otimes_RR_s$ as an $R_s$-module.

\begin{corollary}\label{s-fd-s}
Let $R$ be a ring, $S$ a multiplicative subset of $R$ and $M$ an $R$-module. If $S$-$fd_R(M)\leq n$, then there exists an element $s\in S$ such that $fd_{R_s}(M_s)\leq n$.
\end{corollary}
\begin{proof} Let $M$ be an $R$-module with $S$-$fd_R(M)\leq n$. Then there is an element $s\in S$ such that $s\Tor_{n+1}^R(R/I,M)=0$ for any ideal $I$ of $R$ by Proposition \ref{s-flat d}. Let $I_s$ be an ideal of $R_s$ with $I$ an ideal of $R$. Then  $\Tor_{n+1}^{R_s}(R_s/I_s,M_s)\cong \Tor_{n+1}^{R}(R/I,M)\otimes_RR_s=0$ since $s\Tor_{n+1}^R(R/I,M)=0$. Hence $fd_{R_s}(M_s)\leq n$.

\end{proof}

\begin{corollary}\label{sfd-sfd}
Let $R$ be a ring and $S$ a multiplicative subset of $R$. Suppose $M$ is an $R$-module, then  $S$-$fd_R(M)\geq $$fd_{R_S}M_S$. Moreover, if $S$ is composed of finite elements, then $S$-$fd_R(M)= $$fd_{R_S}M_S$.
\end{corollary}
\begin{proof} Let $...\rightarrow F_n \rightarrow ...\rightarrow F_1\rightarrow F_0\rightarrow M\rightarrow 0$ be an exact sequence with each $F_{i}$ $S$-flat. By localizing at $S$, we can obtain a flat resolution of $M_S$ over $R_S$ as follows:
$$\rightarrow (F_n)_S \rightarrow ...\rightarrow (F_1)_S\rightarrow (F_0)_S\rightarrow (M)_S\rightarrow 0.$$ So  $S$-$fd_R(M)\geq $$fd_{R_S}M_S$ by Proposition \ref{s-flat d}. Suppose $S$ is composed of finite elements and $fd_{R_S}M_S=n$.  Let $0\rightarrow F_{n}\rightarrow F_{n-1} \rightarrow ...\rightarrow F_1\rightarrow F_0\rightarrow M\rightarrow 0$ be an exact sequence, where $F_i$ is flat over $R$ for any $i=0,...,n-1$. Localizing at $S$, we have $(F_n)_S$ is flat over $R_S$. By \cite[Proposition 3.8]{zwz21}, $F$ is $S$-flat. So $S$-$fd_R(M)\leq n$ by Proposition \ref{s-flat d}.
\end{proof}

\begin{proposition}\label{sfd-exact}
Let $R$ be a ring and $S$ a multiplicative subset of $R$. Let $0\rightarrow A\rightarrow B\rightarrow C\rightarrow 0$ be an $S$-exact sequence of $R$-modules. Then the following assertions hold.
\begin{enumerate}
  \item $S$-$fd_R(C)\leq 1+\max\{S$-$fd_R(A),S$-$fd_R(B)\}$.
    \item  If $S$-$fd_R(B)<S$-$fd_R(C)$, then $S$-$fd_R(A)=S$-$fd_R(C)-1>S$-$fd_R(B).$
\end{enumerate}
\end{proposition}
\begin{proof} The proof  is similar with that of the classical case (see \cite[Theorem 3.6.7]{fk16}). So we omit it.
\end{proof}

Let $\p$ be a prime ideal of $R$ and $M$ an $R$-module. We denote by $\p$-$fd_R(M)$ $(R-\p)$-$fd_R(M)$ briefly. The next result gives a new local characterization of flat dimension of an $R$-module.

\begin{proposition}\label{fd-sfd}
Let $R$ be a ring and $M$ an $R$-module. Then
\begin{align*}
fd_R(M)=\sup\{\p\mbox{-}fd_R(M)|\p\in \Spec(R)\}=\sup\{\m\mbox{-}fd_R(M)|\m\in \Max(R)\} .
\end{align*}
\end{proposition}
\begin{proof} Trivially, $\sup  \{\m\mbox{-}fd_R(M) |\ \m\in \Max(R)\} \leq \sup \{\p\mbox{-}fd_R(M) |\ \p\in \Spec(R)\} \leq fd_R(M)$.
Suppose $\sup  \{\m\mbox{-}fd_R(M) |\ \m\in \Max(R)\}=n$. For any $R$-module $N$, there exists an element $s^{\m}\in R-\m$ such that  $s^{\m}\Tor^R_{n+1}(M, N)=0$  by Proposition \ref{s-flat d}. Since the ideal generated by all $s^{\m}$ is $R$, we have $\Tor^R_{n+1}(M, N)=0$ for all $R$-modules $N$. So $fd_R(M)\leq n$. Suppose $\sup  \{\m\mbox{-}fd_R(M) |\ \m\in \Max(R)\}=\infty$. Then for any $n\geq 0$, there exists a maximal ideal $\m$ and an element $s^{\m}\in R-\m$ such that  $s^{\m}\Tor^R_{n+1}(M, N)\not=0$ for some $R$-module $N$. So for any $n\geq 0$, we have $\Tor^R_{n+1}(M, N)\not=0$ for some $R$-module $N$. Thus $fd_R(M)=\infty$. So the equalities hold.
\end{proof}

\section{On the $S$-weak global dimensions of rings}
Recall that the weak global dimension w.gl.dim$(R)$ of a ring $R$ is the supremum of $S$-flat dimensions of all $R$-modules. Now, we introduce the $S$-analogue of weak global dimensions of rings $R$ for a multiplicative subset $S$ of $R$.

\begin{definition}\label{w-phi-flat }
The $S$-weak global dimension of a ring $R$ is defined by
\begin{center}
$S$-w.gl.dim$(R) = \sup\{S$-$fd_R(M) | M $ is an $R$-module$\}.$
\end{center}
\end{definition}\label{def-wML}
Obviously, $S$-w.gl.dim$(R)\leq $w.gl.dim$(R)$ for any multiplicative subset $S$ of $R$.  And if $S$ is composed of units, then $S$-w.gl.dim$(R)=$w.gl.dim$(R)$ .  The next result characterizes the $S$-weak global dimension of a ring $R$.
\begin{proposition}\label{w-g-flat}
Let $R$ be a ring and $S$ a multiplicative subset of $R$. The following statements are equivalent for $R$:
\begin{enumerate}
  \item  $S$-w.gl.dim$(R)\leq  n$;
    \item  $S$-$fd_R(M)\leq n$ for all $R$-modules $M$;
    \item $\Tor^R_{n+k}(M, N)$ is uniformly $S$-torsion for all $R$-modules $M, N$ and all $k > 0$;
    \item  $\Tor^R_{n+1}(M, N)$ is uniformly $S$-torsion for all $R$-modules $M, N$;
     \item there exists an element $s\in S$  such that $s\Tor^R_{n+1}(R/I, R/J)$ for any ideals $I$ and $J$ of $R$.
\end{enumerate}
\end{proposition}
\begin{proof}
$(1) \Rightarrow  (2)$ and $(3)\Rightarrow  (4)$: Trivial.

$(2) \Rightarrow  (3)$: Follows from Proposition \ref{s-flat d}.

$(4) \Rightarrow  (1)$: Let $M$ be an $R$-module and $0 \rightarrow F_n \rightarrow ...\rightarrow F_1\rightarrow F_0\rightarrow M\rightarrow 0$ an exact sequence, where $F_0, F_1, . . . , F_{n-1}$ are flat $R$-modules.
To complete the proof, it suffices, by Proposition \ref{s-flat d}, to prove that $F_n$ is
$S$-flat. Let $N$ be an $R$-module.
Thus $S$-$fd_R(N)\leq n$ by (4). It follows from Corollary \ref{big-Tor} that $\Tor^R_1 (N, F_n)\cong \Tor^R_{n+1}(N, M)$ is uniformly $S$-torsion. Thus $F_n$ is $S$-flat.

$(4)\Rightarrow  (5)$: Let $M=\bigoplus\limits_{I\unlhd R} R/I$ and $N=\bigoplus\limits_{J\unlhd R} R/J$. Then there exists $s\in S$ such that $$s\Tor^R_{n+1}(M, N)=s\bigoplus\limits_{I\unlhd R, J\unlhd R}\Tor_{n+1}^R(R/I,R/J)=0.$$
Thus $s\Tor_{n+1}^R(R/I,R/J)=0$ for any ideals $I,J$ of $R$.

$(5)\Rightarrow  (4)$: Suppose $M$ is generated by $\{m_i|i\in \Gamma\}$ and $N$ is generated by $\{n_i|i\in \Lambda\}$.  Set $M_0=0$ and $M_\alpha=\langle m_i|i<\alpha\rangle$  for each $\alpha\leq \Gamma$.  Then $M$ have a continuous   filtration  $\{M_\alpha|\alpha\leq \Gamma \}$ with $M_{\alpha+1}/M_\alpha\cong R/I_{\alpha+1} $ and $I_{\alpha}=\Ann_R(m_{\alpha}+M_\alpha\cap Rm_{\alpha})$. Similarly,  $N$ have a continuous   filtration  $\{N_\beta|\beta\leq \Lambda \}$ with $N_{\beta+1}/N_\beta\cong R/J_{\beta+1} $ and $J_{\beta}=\Ann_R(n_{\beta}+N_\beta\cap Rn_{\beta})$.  Since $s\Tor_{n+1}^R(R/I_{\alpha},R/J_{\beta})=0$ for each $\alpha\leq \Gamma$ and $\beta\leq \Lambda$, it is easy to verify $s\Tor_{n+1}^R(M,N)=0$ by transfinite induction on both positions of $M$ and $N$.
\end{proof}

The following Corollaries \ref{s-wgd-s}, \ref{s-swd-swd}, \ref{swd-swd} and \ref{wgld-swgld} can be deduced by Corollaries \ref{s-fd-s}, \ref{sfd-sfd}, \ref{sfd-sfd-1} and Proposition \ref{fd-sfd}.

\begin{corollary}\label{swd-swd}
Let $R$ be a ring and $S'\subseteq S$  multiplicative subsets of $R$. Then  $S$-w.gl.dim$(R) \leq S'$-w.gl.dim$(R)$.
\end{corollary}

\begin{corollary}\label{s-wgd-s}
Let $R$ be a ring and $S$ a multiplicative subset of $R$. If $S$-w.gl.dim$(R) \leq n$, then there exists an element $s\in S$ such that w.gl.dim$(R_s) \leq n$.
\end{corollary}

\begin{corollary}\label{s-swd-swd}
Let $R$ be a ring and $S$ a multiplicative subset of $R$. Then  $S$-w.gl.dim$(R)\leq$ w.gl.dim$(R_S)$. Moreover, if $S$ is composed of finite elements, then $S$-w.gl.dim$(R)=$w.gl.dim$(R_S)$.
\end{corollary}

The following example shows that the converse of Corollary \ref{s-swd-swd} is not  true in general.
\begin{example}\label{s-wgld-1-not-wgld-1}
Let $R=k[x_1,x_2,...,x_{n+1}]$ be a polynomial ring with $n+1$ indeterminates over a field $k$ $(n\geq 0)$. Set $S=k[x_1]-\{0\}$. Then $S$ is a multiplicative subset of $R$ and $R_S=k(x_1)[x_2,...,x_{n+1}]$ is a polynomial ring with $n$ indeterminates over the field $k(x_1)$. So w.gl.dim$(R_S)=n$ by \cite[Theorem 3.8.23]{fk16}. Let  $s\in S$, we have $R_s=k[x_1]_s[x_2,...,x_{n+1}]$. Since $k[x_1]$ is not a G-domain, $k[x_1]_s$ is not a field $($see \cite[Theorem 21]{K70}$)$. Thus w.gl.dim$(k[x_1]_s)=1$. So w.gl.dim$(R_s)=n+1$ for any $s\in S$ by  \cite[Theorem 3.8.23]{fk16} again. Consequently  $S$-w.gl.dim$(R)\geq n+1$ by Corollary \ref{s-wgd-s}.
\end{example}

Let $\p$ be a prime ideal of a ring $R$ and $\p$-$w.gl.dim(R)$ denote $(R-\p)$-$w.gl.dim(R)$ briefly.
We have a new local characterization of  weak global dimensions of commutative rings.
\begin{corollary}\label{wgld-swgld}
Let $R$ be a ring. Then
\begin{align*}
&w.gl.dim(R)=\sup\{\p\mbox{-}w.gl.dim(R)|\p\in \Spec(R)\}=\sup\{\m\mbox{-}w.gl.dim(R)|\m\in \Max(R)\} .
\end{align*}
\end{corollary}

The rest of this this section mainly consider rings with $S$-weak global dimensions at most one. Recall from \cite{zwz21} that a ring $R$ is called  $S$-von Neumann regular  provided that there exists  $s\in S$ such that  for any $a\in R$ there exists  $r\in R$ such that $sa=ra^2$. Thus by \cite[Theorem 3.11]{zwz21}, the following result holds.
\begin{corollary}\label{s-vn-ext-char}
Let $R$ be a ring and $S$ a multiplicative subset of $R$. The following assertions are equivalent:
\begin{enumerate}
\item  $R$ is an $S$-von Neumann regular ring;
\item    for any $R$-module $M$ and $N$, there exists  $s\in S$ such that $s\Tor_1^R(M,N)=0$;
\item  there exists  $s\in S$ such that $s\Tor_1^R(R/I,R/J)=0$    for any ideals $I$ and $J$ of $R$;
\item  any $R$-module is $R$-flat;
\item  $S$-w.gl.dim$(R)=0$.
\end{enumerate}
\end{corollary}

Trivially, von Neumann regular rings are $S$-von Neumann regular, and if a ring $R$ is $S$-von Neumann regular ring then $R_S$ is von Neumann regular. It was proved in \cite[Proposition 3.17]{zwz21} that if the multiplicative subset $S$ of $R$ is composed of non-zero-divisors, then $R$ is $S$-von Neumann regular if and only if  $R$ is  von Neumann regular. Examples of $S$-von Neumann regular rings that are not  von Neumann regular, and a ring $R$ for which $R_S$ is  von Neumann regular but $R$ is not $S$-von Neumann regular are given in  \cite{zwz21}.

\begin{proposition}\label{s-wgld-1}
Let $R$ be a ring and $S$ a multiplicative subset of $R$. The following assertions are equivalent:
\begin{enumerate}
\item  $S$-w.gl.dim$(R)\leq 1$;
\item   any submodule of $S$-flat modules is $S$-flat;
\item   any submodule of flat modules is $S$-flat;
\item   $\Tor^R_{2}(M, N)$ is uniformly $S$-torsion for all $R$-modules $M, N$;
\item   there exists an element $s\in S$ such that $s\Tor_2^R(R/I,R/J)=0$ for any ideals $I,J$ of $R$.
\end{enumerate}
\end{proposition}

\begin{proof}
The equivalences follow from Proposition \ref{w-g-flat}.
\end{proof}

The following lemma can be found in \cite[Chapter 1 Exercise 6.3]{FS01} for integral domains. However it is also true for any commutative rings and we give a proof for completeness.
\begin{lemma}\label{tor-2}
Let $R$ be a ring and $I,J$ ideals of $R$, then $\Tor_2^R(R/I,R/J)\cong\Ker(\phi)$ where $\phi:I\otimes J\rightarrow IJ$ is an $R$-homomorphism, where  $\phi(a\otimes b)=ab$.
\end{lemma}
\begin{proof} Let $I$ and $J$ be ideals of $R$, then $\Tor_2^R(R/I,R/J)\cong \Tor_1^R(R/I,J)$. Consider the following exact sequence:
$0\rightarrow \Tor_1^R(R/I,J)\rightarrow I\otimes_RJ\xrightarrow{\phi} R\otimes_R J$, where $\phi$ is an $R$-homomorphism such that   $\phi(a\otimes b)=ab$. We have $\Tor_2^R(R/I,R/J)\cong\Ker(\phi)$.
\end{proof}


Trivially,  a ring $R$ with w.gl.dim$(R)\leq 1$ has $S$-weak global dimension at most one. The following example shows the converse does not hold generally.

\begin{example}\label{s-wgld-1-not-wgld-1}
Let $A$ be a ring with w.gl.dim$(A)= 1$, $T=A\times A$  the direct product of $A$. Set $s=(1,0)\in T$, then $s^2=s$. Let $R=T[x]/\langle sx,x^2\rangle$ with $x$ an indeterminate  and $S=\{1,s\}$ be a multiplicative subset of $R$. Then $S$-w.gl.dim$(R)=1$ but w.gl.dim$(R)=\infty$.
\end{example}
\begin{proof} Note that every element in $R$ can be written as $r=(a,b)+(0,c)x$ where $a,b,c\in A$.
 Let $f:R\rightarrow A$  be a ring homomorphism where $f((a,b)+(0,c)x)=a$. Then $f$ makes $A$ a module retract of $R$. Let $I$ and $J$ be ideals of $R$. Suppose $r_1=(a_1,b_1)+(0,c_1)x$ and $r_2=(a_2,b_2)+(0,c_2)$ are elements in $I$ and $J$ respectively such that  $r_1\otimes r_2\in \Ker(\phi)$,  where $\phi:I\otimes_R J\rightarrow IJ$ is the multiplicative homomorphism. Then $r_1r_2=(a_1a_2,b_1b_2)+(0,b_1c_2+b_2c_1)x=0$, so $a_1a_2=0$ in $A$.  By Lemma \ref{tor-2}, $a_1\otimes_Aa_2=0$ in $f(I)\otimes_Af(J)$ since w.gl.dim$(A)=1$. Consequently $s^2r_1\otimes_R r_2=sr_1\otimes_R sr_2=(a_1,0)\otimes_R(a_2,0)=0$ in $I\otimes J$. So $s^2\Tor_2^R(R/I,R/J)=0$ by Lemma \ref{tor-2}. It follows that $S$-w.gl.dim$(R)\leq 1$ by Proposition \ref{s-wgld-1}. Since $R_S\cong A$ have weak global dimension $1$,  $S$-w.gl.dim$(R)= 1$ by Corollary \ref{s-vn-ext-char} and \cite[Corollary 3.14]{zwz21}. Since $R$ is non-reduced coherent ring, then w.gl.dim$(R)=\infty$ by \cite[Corollary 4.2.4]{g}.
\end{proof}

\section{$S$-weak global dimensions of factor rings and  polynomial rings}

In this section, we mainly consider the $S$-weak global dimensions of factor rings and  polynomial rings. Firstly, we give an inequality of $S$-weak global dimensions for ring homomorphisms. Let $\theta:R\rightarrow T$ be a ring homomorphism. Suppose $S$ is a multiplicative subset of $R$, then $\theta(S)=\{\theta(s)|s\in S\}$  is a multiplicative subset of $T$.

\begin{lemma}\label{tor-s-poly-1}
Let $\theta:R\rightarrow T$ be a ring homomorphism, $S$ a multiplicative subset of $R$. Suppose $L$ is a $\theta(S)$-flat $T$-module. Then for any $R$-module $X$ and any $n\geq 0$,  $\Tor_n^R(X,L)$ is $S$-isomorphic to $\Tor_n^R(X,T)\otimes_TL.$ Consequently, $S$-$fd_{R}(L)\leq S$-$fd_{R}(T)$.
\end{lemma}
\begin{proof} If $n=0$, then $X\otimes_RL\cong X\otimes_R(T\otimes_TL)\cong(X\otimes_RT)\otimes_TL.$

If $n=1$, let $0\rightarrow A\rightarrow P\rightarrow X\rightarrow 0$ be an exact sequence of $R$-modules where $P$ is free. Thus we have two exact sequences of $T$-module: $0\rightarrow \Tor_1^R(X,T)\rightarrow A\otimes_RT\rightarrow P\otimes_RT\rightarrow X\otimes_RT\rightarrow 0$ and  $0\rightarrow \Tor_1^R(X,L)\rightarrow A\otimes_RL\rightarrow P\otimes_RL\rightarrow X\otimes_RL\rightarrow 0.$ Consider the following commutative diagram with exact sequence:
$$\xymatrix@R=20pt@C=20pt{
0\ar[r]^{}  & 0 \ar[r]^{} \ar[d]^{} & \Tor_1^R(X,L)\ar[d]_{h}\ar[r]^{} &A\otimes_RL \ar[d]^{\cong}\ar[r]^{} &P\otimes_RL\ar[d]^{\cong} \\
0\ar[r]^{} &\Ker(\delta)\ar[r]^{} &\Tor_1^R(X,T)\otimes_TL\ar[r]^{\delta}  & (A\otimes_RT)\otimes_TL  \ar[r]^{} & (P\otimes_RT)\otimes_TL .\\}$$
Since $L$ is a $\theta(S)$-flat $T$-module, $\delta$ is a $\theta(S)$-monomorphism. By Theorem \ref{s-5-lemma}, $h$ is a $\theta(S)$-isomorphism over $T$. So $h$ is an $S$-isomorphism over $R$ since $T$-modules are viewed as $R$-modules through $\theta$. By dimension-shifting, we can obtain that $\Tor_n^R(X,L)$ is $S$-isomorphic to $\Tor_n^R(X,T)\otimes_TL$  for any $R$-module $X$ and any $n\geq 0$.

Thus for any $R$-module $X$, if $\Tor_n^R(X,T)$ is uniformly $S$-torsion, then $\Tor_n^R(X,L)$  is also uniformly $S$-torsion. Consequently, $S$-$fd_{R}(L)\leq S$-$fd_{R}(T)$.
\end{proof}

\begin{proposition}\label{sfd-changring}
Let $\theta:R\rightarrow T$ be a ring homomorphism, $S$ a multiplicative subset of $R$. Suppose $M$ is an $T$-module. Then
\begin{center}
$S$-$fd_{R}(M)\leq\theta(S)$-$fd_{T}(M)+S$-$fd_{R}(T).$
\end{center}
\end{proposition}
\begin{proof} Assume $\theta(S)$-$fd_{T}(M)=n<\infty$. If $n=0$, then $M$ is  $\theta(S)$-flat  over $T$. By Lemma \ref{tor-s-poly-1}, $S$-$fd_{R}(M)\leq n+S$-$fd_{R}(T)$.

Now we assume $n>0$. Let $0\rightarrow A\rightarrow F\rightarrow M\rightarrow 0$ be an exact sequence of $T$-modules, where $F$ is a free $T$-module. Then $\theta(S)$-$fd_{T}(A)=n-1$ by Corollary \ref{big-Tor} and Proposition \ref{s-flat d}. By induction, $S$-$fd_{R}(A)\leq n-1+S$-$fd_{R}(T)$. Note that $S$-$fd_{R}(T)=S$-$fd_{R}(F)$. By Proposition \ref{sfd-exact}, we have
\begin{align*}
 S\mbox{-}fd_{R}(M)&\leq 1+\max\{S\mbox{-}fd_{R}(F), S\mbox{-}fd_{R}(A)\} \\
  &\leq 1+n-1+S\mbox{-}fd_{R}(T) \\
   &=\theta(S)\mbox{-}fd_{T}(M)+S\mbox{-}fd_{R}(T).
\end{align*}
\end{proof}

Let $R$ be a ring, $I$ an ideal of $R$ and  $S$  a multiplicative subset of $R$. Then $\pi:R\rightarrow R/I$ is a ring epimorphism and $\pi(S):=\overline{S}=\{s+I\in R/I|s\in S\}$ is naturally a multiplicative subset of $R/I$.

\begin{proposition}\label{s-fd-poly-3}
Let $R$ be a ring, $S$ a multiplicative subset of $R$. Let $a\in R$ be neither a zero-divisor nor a unit. Written $\overline{R}=R/aR$ and $\overline{S}=\{s+aR\in \overline{R}|s\in S\}$. Then the following assertions hold.
\begin{enumerate}
\item  Let $M$ be a nonzero $\overline{R}$-module. If  $\overline{S}$-$fd_{\overline{R}}(M)<\infty$, then
\begin{center}
$S$-$fd_{R}(M)=\overline{S}$-$fd_{\overline{R}}(M)+1.$
\end{center}
\item   If $\overline{S}$-w.gl.dim$(\overline{R})<\infty$, then
\begin{center}
$S$-w.gl.dim$(R) \geq \overline{S}$-w.gl.dim$(\overline{R})+1$.
\end{center}
\end{enumerate}
\end{proposition}
\begin{proof} (1) Set $\overline{S}$-$fd_{\overline{R}}(M)=n$. Since $a\in R$ be neither a zero-divisor nor a unit, then $S$-$fd_R(\overline{R})=1$. By Proposition \ref{sfd-changring}, we have $S$-$fd_R(M)\leq \overline{S}$-$fd_{\overline{R}}(M)+1=n+1.$ Since  $\overline{S}$-$fd_{\overline{R}}(M)=n$, then there is an injective $\overline{R}$-module $C$ such that $\Tor_n^{\overline{R}}(M,C)$ is not uniformly $\overline{S}$-torsion. By \cite[Theorem 2.4.22]{fk16}, there is an injective $R$-module $E$ such that $0\rightarrow C\rightarrow E\rightarrow E\rightarrow 0$ is exact. By \cite[Proposition 3.8.12(4)]{fk16}, $\Tor_{n+1}^{R}(M,E)\cong \Tor_n^{\overline{R}}(M,C)$. Thus $\Tor_{n+1}^{R}(M,E)$ is not uniformly $S$-torsion. So $S$-$fd_{R}(M)=\overline{S}$-$fd_{\overline{R}}(M)+1$.

(2) Let $n=\overline{S}$-w.gl.dim$(\overline{R})$. Then there is a nonzero $\overline{R}$-module $M$ such that $\overline{S}$-$fd_{\overline{R}}(M)=n$. Thus $S$-$fd_{R}(M)=n+1$ by (1). So $S$-w.gl.dim$(R) \geq \overline{S}$-w.gl.dim$(\overline{R})+1$.
\end{proof}

Let $R$ be a ring and $M$ an $R$-module. $R[x]$ denotes the  polynomial ring with one indeterminate, where all coefficients are in $R$. Set $M[x]=M\otimes_RR[x]$, then $M[x]$ can be seen as an $R[x]$-module naturally. It is well-known w.gl.dim$(R[x])=$w.gl.dim$(R)$ (see \cite[Theorem 3.8.23]{fk16}). In this section, we give a $S$-analogue of this result. Let $S$ be a multiplicative subset of $R$, then $S$ is a multiplicative subset of $R[x]$ naturally.

\begin{lemma}\label{tor-s-poly}
Let $R$ be a ring, $S$ a multiplicative subset of $R$. Suppose $T$ is an $R$-module and  $F$ is an $R[x]$-module. Then the following assertions hold.
\begin{enumerate}
\item   $T$ is  uniformly $S$-torsion over $R$ if and only if $T[x]$ is a uniformly $S$-torsion $R[x]$-module.
\item  If  $F$ is $S$-flat over $R[x]$, then $F$ is  $S$-flat over $R$.
\end{enumerate}
\end{lemma}
\begin{proof} $(1)$ If $sT[x]=0$ for some $s\in S$, then trivially $sT=0$. So  $T$ is  uniformly $S$-torsion over $R$. Suppose $sT=0$ for some $s\in S$. Then $sT[x]\cong (sT)[x]=0$. Thus $T[x]$ is a uniformly $S$-torsion $R[x]$-module.

$(2)$ Suppose $F$ is an $S$-flat $R[x]$-module. By \cite[Theorem 1.3.11]{g}, $\Tor^R_{1}(F,L)[x]\cong \Tor^{R[x]}_{1}(F[x],L[x])$ is uniformly $S$-torsion. Thus there exists an element $s\in S$ such that $s\Tor^R_{1}(F,L)[x]=0$. So $s\Tor^R_{1}(F,L)=0$. It follows that $F$ is an $S$-flat $R$-module.
\end{proof}

\begin{proposition}\label{s-fd-poly}
Let $R$ be a ring, $S$ a multiplicative subset of $R$ and $M$ an $R$-module. Then $S$-$fd_{R[x]}(M[x])=S$-$fd_R(M)$.
\end{proposition}
\begin{proof}

Assume that $S$-$fd_R(M)\leq n$. Then $\Tor^R_{n+1}(M,N)$ is uniformly $S$-torsion for any $R$-module $N$ by Proposition \ref{s-flat d}. Thus for any $R[x]$-module $L$, $\Tor^{R[x]}_{n+1}(M[x],L)\cong \Tor^R_{n+1}(M,L)$  is uniformly $S$-torsion for any $R[x]$-module $L$ by \cite[Theorem 1.3.11]{g}. Consequently, $S$-$fd_{R[x]}(M[x])\leq n$ by Proposition \ref{s-flat d}.

Let $0\rightarrow F_n \rightarrow ...\rightarrow F_1\rightarrow F_0\rightarrow M[x]\rightarrow 0$ be an exact sequence with each $F_i$ $S$-flat over $R[x]$ ($1\leq i\leq n$). Then it is also $S$-flat resolution of $M[x]$ over $R$  by Lemma \ref{tor-s-poly}. Thus $\Tor_{n+1}^R(M[x],N)$ is  uniformly $S$-torsion for any $R$-module $N$ by Proposition \ref{s-flat d}. It follows that  $s\Tor_{n+1}^R(M[x],N)=s \bigoplus\limits_{i=1}^{\infty} \Tor_{n+1}^R(M,N)=0 $. Thus $\Tor_{n+1}^R(M,N)$ is  uniformly $S$-torsion. Consequently, $S$-$fd_R(M)\leq S$-$fd_{R[x]}(M[x])$ by Proposition \ref{s-flat d} again.
\end{proof}

Let $M$ be an $R[x]$-module then $M$ can be naturally viewed as an $R$-module. Define $\psi:M[x]\rightarrow M$ by $$
\psi(\sum\limits_{i=0}^nx^i\otimes m_i)=\sum\limits_{i=0}^nx^i m_i,\qquad m_i\in M.$$ And define $\varphi:M[x]\rightarrow M[x]$ by $$\varphi(\sum\limits_{i=0}^nx^i\otimes m_i)=\sum\limits_{i=0}^nx^{i+1}\otimes m_i-\sum\limits_{i=0}^nx^i\otimes xm_i,\qquad m_i\in M.$$

\begin{lemma}\cite[Theorem 3.8.22]{fk16}\label{exact-s-poly}
Let $R$ be a ring, $S$ a multiplicative subset of $R$. For any $R[x]$-module $M$, $$0\rightarrow M[x]\xrightarrow{\varphi} M[x]\xrightarrow{\psi} M\rightarrow 0$$
is exact.
\end{lemma}

\begin{theorem}\label{s-wgd-poly}
Let $R$ be a ring, $S$ a multiplicative subset of $R$. Then $S$-w.gl.dim$(R[x])=S$-w.gl.dim$(R)+1$.
\end{theorem}
\begin{proof} Let $M$ be an $R[x]$-module. Then, by Lemma \ref{exact-s-poly},  there is an exact sequence over $R[x]$:
$$0\rightarrow M[x]\rightarrow M[x]\rightarrow M\rightarrow 0.$$
By Proposition \ref{sfd-exact} and Proposition \ref{s-fd-poly},
\begin{center}
 $S$-$fd_R(M)\leq S$-$fd_{R[x]}(M)\leq 1+S$-$fd_{R[x]}(M[x])=1+S$-$fd_R(M)\qquad \qquad (\ast)$.
\end{center}
Thus if $S$-w.gl.dim$(R)< \infty$, then $S$-w.gl.dim$(R[x])<\infty$.

Conversely, if $S$-w.gl.dim$(R[x])< \infty$, then for any $R$-module $M$, $S$-$fd_R(M)=S$-$fd_{R[x]}(M[x])< \infty$ by Proposition \ref{s-fd-poly}. Therefore we have $S$-w.gl.dim$(R)< \infty$ if and only if $S$-w.gl.dim$(R[x])< \infty$.
Now we assume that both of these are finite. Then $S$-w.gl.dim$(R[x])\leq S$-w.gl.dim$(R)+1$ by $(\ast)$. Since $R\cong R[x]/xR[x]$, $S$-w.gl.dim$(R[x]) \geq S$-w.gl.dim$(R)+1$ by Proposition \ref{s-fd-poly-3}. Consequently, we have $S$-w.gl.dim$(R[x])=S$-w.gl.dim$(R)+1$.
\end{proof}
\begin{corollary}\label{s-wgd-poly-duo}
Let $R$ be a ring, $S$ a multiplicative subset of $R$. Then for any $n\geq 1$ we have
\begin{center}
$S$-w.gl.dim$(R[x_1,...,x_n])=S$-w.gl.dim$(R)+n$.
\end{center}
\end{corollary}

\end{document}